\documentclass[10pt]{article}

\usepackage[hmargin=35mm,vmargin=40mm]{geometry}

\usepackage{amssymb}           % For mathematical symbols
\usepackage{amsthm}            % For custom theorem styles
\usepackage{enumerate}         % To support \begin{enumerate}[(i)]
\usepackage{graphicx}          % For embedding figures
\usepackage[tight]{subfigure}  % For embedding subfigures

\newtheorem{theorem}{Theorem}
\newtheorem{lemma}[theorem]{Lemma}

\newtheorem{algorithm}[theorem]{Algorithm}

\newtheorem{test}[theorem]{Test}
\newtheorem{conjecture}{Conjecture}

\newcommand{\ppirr}{{$\mathbb{P}^2$-ir\-re\-du\-ci\-ble}}
\newcommand{\R}{\mathbb{R}}
\newcommand{\regina}{\emph{Regina}}
\newcommand{\tri}{\mathcal{T}}

\title{Detecting genus in vertex links for the \\
    fast enumeration of 3-manifold triangulations}
\author{Benjamin A.~Burton}
\date{January 16, 2011}

\begin{document}

\maketitle

\begin{abstract}
Enumerating all 3-manifold triangulations of a given size is a
difficult but increasingly important problem in computational topology.
A key difficulty for enumeration algorithms is that most combinatorial
triangulations must be discarded because they do not represent
topological 3-manifolds.
In this paper we show how to preempt bad triangulations
by detecting genus in partially-constructed vertex links,
allowing us to prune the enumeration tree substantially.

The key idea is to manipulate the boundary edges surrounding partial vertex
links using expected logarithmic time operations.
Practical testing shows the
resulting enumeration algorithm to be significantly faster,
with up to $249 \times$ speed-ups even for small problems where
comparisons are feasible.
We also discuss parallelisation, and
describe new data sets that have been obtained using high-performance
computing facilities.

\medskip
\noindent \textbf{ACM classes}\quad
G.2.1; G.4; I.1.2

\end{abstract}

%%%%%%%%%%%%%%%%%%%%%%%%%%%%%%%%%%%%%%%%%%%%%%%%%%%%%%%%%%%%%%%%%%%%%%%%
%
%   Section:  Introduction
%
%%%%%%%%%%%%%%%%%%%%%%%%%%%%%%%%%%%%%%%%%%%%%%%%%%%%%%%%%%%%%%%%%%%%%%%%

\section{Introduction} \label{s-intro}

In computational geometry and topology,
triangulations are natural and ubiquitous data structures for
representing topological spaces.
Here we focus on triangulations in 3-manifold topology,
an important branch of topology in which many key problems
are theoretically decidable but extremely difficult for practical
computation \cite{agol02-knotgenus,burton10-complexity}.

A \emph{census} of 3-manifold triangulations is a list of
all triangulations that satisfy some given set of properties.
A typical census fixes the number of tetrahedra (the \emph{size}
of the triangulation), and enumerates all triangulations up to
\emph{isomorphism} (a relabelling of the tetrahedra and their vertices).

Censuses of this type first appeared in the late 1980s
\cite{hildebrand89-cuspedcensusold,matveev88-hamiltonian}.\footnote{%
    Some authors, following Matveev \cite{matveev88-hamiltonian},
    work in the setting of \emph{special spines}
    which are dual to triangulations.}
These were censuses of \emph{minimal triangulations}, which
represent a given 3-manifold using the fewest possible
tetrahedra.
Such censuses have brought about
new insights into the combinatorics of minimal triangulations
\cite{burton07-nor8,martelli04-families,matveev98-or6}
and the complexities of 3-manifolds
\cite{martelli02-decomp,matveev03-algms},
and have proven useful for computation and experimentation
\cite{budney08-emb11,burton09-convert}.

A more recent development has been censuses of \emph{all} possible
3-manifold triangulations of a given size, including
non-minimal triangulations \cite{burton10-complexity}.
These have yielded surprising experimental insights into algorithmic
complexity problems and random 3-manifold triangulations
\cite{burton10-complexity,burton10-pachner}, both topics which remain
extremely difficult to handle theoretically.

The limits of censuses in the literature are fairly small.
For closed {\ppirr} 3-manifolds,
all minimal triangulations have been enumerated for only
$\leq 10$ tetrahedra \cite{burton07-nor10};
for the orientable case only, the list of manifolds (but not
triangulations) is known for $\leq 12$ tetrahedra \cite{matveev05-or12}.
A full list of all closed 3-manifold triangulations (non-minimal included)
is known for just $\leq 9$ tetrahedra \cite{burton10-complexity}.
These small limits are unavoidable because censuses grow
exponentially---and sometimes super-exponentially
\cite{burton10-pachner}---in size.

Nevertheless, in theory it should be possible to extend these results
substantially.  Census algorithms are typically based on a recursive
search through all possible \emph{combinatorial triangulations}---that
is, methods of gluing together faces of tetrahedra
in pairs.  However, as the number of tetrahedra grows
large, almost all combinatorial triangulations are \emph{not} 3-manifold
triangulations \cite{dunfield06-random-covers}.  The problem is that
\emph{vertex links}---boundaries of small neighbourhoods of the
vertices of a triangulation---are generally not spheres or discs as they
should be, but instead higher-genus surfaces.

To illustrate: for $n=9$ tetrahedra, if we simply glue together all
$4n$ tetrahedron faces in pairs, a very rough estimate (described in
the appendix) gives at least $6.44 \times 10^{12}$
connected \emph{combinatorial} triangulations up to isomorphism.
However, just 139\,103\,032 are \emph{3-manifold}
triangulations \cite{burton10-complexity}, and a mere 3\,338 are
minimal triangulations of closed {\ppirr} 3-manifolds \cite{burton07-nor10}.

It is clear then that large branches of the combinatorial
search tree can be avoided,
if one could only identify \emph{which} branches these are.
The current challenge for enumeration algorithms is to find
new and easily-testable conditions under which such branches
can be pruned.

For censuses of minimal triangulations, several such conditions
are known: examples include the absence of low-degree edges
\cite{burton04-facegraphs,hildebrand89-cuspedcensusold,matveev98-or6},
or of ``bad subgraphs'' in the underlying
4-valent face pairing graph \cite{burton07-nor10,martelli01-or9}.
Nevertheless, seeing how few minimal triangulations are found in practice,
there are likely many more conditions yet to be found.

It is critical that such conditions can be tested quickly, since these tests
are run on a continual basis as the search progresses and backtracks.
The paper \cite{burton07-nor10} introduces a modified union-find framework
through which several minimality tests can be performed in $O(\log n)$ time.
Importantly, this framework can also be used with censuses of
\emph{all} 3-manifold triangulations (non-minimal included), where it is used
to test that:
(i)~partially-constructed vertex links are orientable,
and (ii)~fully-constructed edges are not identified with themselves in
reverse.
Although tests (i) and (ii) are powerful when enumerating
\emph{non-orientable} triangulations, they do not help for
\emph{orientable} triangulations because they
are already enforced by the overall enumeration algorithm.

The main contributions of this paper are:
\begin{itemize}
    \item We add the following condition to the suite of tests
    used during enumeration:
    \emph{all partially-constructed vertex links must be punctured spheres}
    (not punctured higher-genus surfaces).

    Although this condition is straightforward, it has traditionally
    required $O(n)$ operations to test, making it
    impractical for frequent use in the enumeration algorithm.
    In Section~\ref{s-links} we show how to test this condition
    \emph{incrementally},
    using only expected logarithmic-time operations at each stage of the
    combinatorial search.

    This condition is extremely powerful
    in both the orientable and non-orientable cases,
    and our new incremental test makes it practical for real use.
    Performance testing in Section~\ref{s-perf} shows speed-ups of up to
    $249\times$ even for small enumeration problems where
    experimental comparisons are feasible ($n \leq 7$).

    \item In Section~\ref{s-data} we use this
    to obtain new census data, including
    all closed 3-manifold triangulations of size
    $n \leq 10$ (non-minimal included, improving the previous limit of
    $n \leq 9$), and
    all minimal triangulations of closed {\ppirr}
    3-manifolds of size $n \leq 11$ (improving the previous limit of
    $n \leq 10$).
    High-performance computing and distributed algorithms play a key
    role, as outlined in Section~\ref{s-par}.
    All censuses cover both orientable and non-orientable triangulations.
\end{itemize}

This new census data is already proving useful in ongoing projects,
such as studying the structure of minimal triangulations, and mapping out
average-case and generic complexities for difficult topological
decision problems.

It should be noted that avoiding isomorphisms---often a significant
difficulty in combinatorial enumeration---is not a problem here.
See the full version of this paper for details.

Looking forward, the techniques of this paper can also be applied to the
enumeration of \emph{4-manifold} triangulations.
In this higher-dimensional setting, our techniques can be applied to
edge links rather than vertex links.
Vertices on the other hand become more difficult to handle: each vertex link
must be a 3-sphere, and 3-sphere recognition remains a difficult
algorithmic problem \cite{burton10-pachner,matveev03-algms}.
Here research into algebraic techniques may yield
new heuristics to further prune the search tree.

Throughout this paper we restrict our attention to closed
3-mani\-folds,
although all of the results presented here extend easily to manifolds
with boundary.

All algorithms described in this paper can be downloaded as
part of {\regina} \cite{regina,burton04-regina}, an open-source software
package for the algebraic and combinatorial manipulation of 3-manifolds and
their triangulations.

%%%%%%%%%%%%%%%%%%%%%%%%%%%%%%%%%%%%%%%%%%%%%%%%%%%%%%%%%%%%%%%%%%%%%%%%
%
%   Section:  Preliminaries
%
%%%%%%%%%%%%%%%%%%%%%%%%%%%%%%%%%%%%%%%%%%%%%%%%%%%%%%%%%%%%%%%%%%%%%%%%

\section{Preliminaries} \label{s-prelim}

Consider a collection of $n$ tetrahedra (these are abstract objects, and
need not be embedded in some $\R^d$).  A \emph{combinatorial
triangulation} of \emph{size $n$} is obtained by affinely identifying
(or ``gluing together'') the $4n$ tetrahedron faces in pairs.\footnote{%
    A combinatorial triangulation need not be a simplicial complex,
    and need not represent a topological 3-manifold.  The word
    ``combinatorial'' indicates that we are only interested in face
    identifications, with no topological requirements.}
Specifically, it consists of the following data:
\begin{itemize}
    \item a partition of the $4n$ tetrahedron faces into $2n$ pairs,
    indicating which faces are to be identified;
    \item $2n$ permutations of three elements, indicating which of the six
    possible rotations or reflections is to be used for each
    identification.
\end{itemize}

For instance, consider the following example with $n=3$ tetrahedra.
The tetrahedra are labelled $A,B,C$, and the four vertices of each
tetrahedron are labelled $0,1,2,3$.
\[ \small \begin{array}{c|r|r|r|r}
\mathrm{Tetrahedron} & \mathrm{Face}\ 012 & \mathrm{Face}\ 013 &
\mathrm{Face}\ 023 & \mathrm{Face}\ 123 \\
\hline
A & C: 013 & B: 012 & A: 312 & A: 230 \\
B & A: 013 & C: 120 & C: 231 & C: 302 \\
C & B: 301 & A: 012 & B: 231 & B: 302
\end{array} \]
The top-left cell of this table indicates that face $012$ of tetrahedron~$A$
is identified with face $013$ of tetrahedron~$C$, using the rotation or
reflection that maps vertices $0,1,2$ of tetrahedron~$A$ to vertices
$0,1,3$ of tetrahedron~$C$ respectively.  For convenience, the same
identification is also shown from the other direction in the second cell
of the bottom row.

As a consequence of these face identifications, we find that several
tetrahedron edges become identified together; each such equivalence
class is called an \emph{edge of the triangulation}.  Likewise, each
equivalence class of identified vertices is called a \emph{vertex of the
triangulation}.  The triangulation illustrated above has three edges and
just one vertex.

The \emph{face pairing graph} of a combinatorial triangulation is the
4-valent multigraph whose nodes represent tetrahedra and whose edges
represent face identifications.  The face pairing graph for the example
above is shown in Figure~\ref{fig-grapheg}.
A combinatorial triangulation is called \emph{connected} if and only if
its face pairing graph is connected.

\begin{figure}[htb]
    \centering
    \includegraphics[scale=1.0]{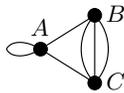}
    \caption{An example face pairing graph}
    \label{fig-grapheg}
\end{figure}

The \emph{vertex links} of a triangulation are obtained as
follows.  In each tetrahedron we place four triangles surrounding the
four vertices, as shown in Figure~\ref{fig-linktri}.  We then glue
together the edges of these triangles in a manner consistent with the
face identifications of the surrounding tetrahedra, as illustrated
in Figure~\ref{fig-linkglue}.

\begin{figure}[htb]
    \centering
    \includegraphics[scale=0.5]{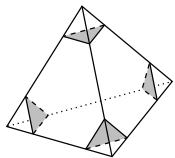}
    \caption{The triangles that form the vertex links}
    \label{fig-linktri}
\end{figure}

\begin{figure}[htb]
    \centering
    \includegraphics[scale=0.5]{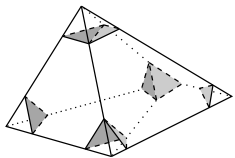}
    \caption{Joining vertex linking triangles along their edges
        in adjacent tetrahedra}
    \label{fig-linkglue}
\end{figure}

The result is a collection of triangulated closed surfaces, one surrounding
each vertex of the triangulation.  The surface surrounding vertex
$V$ is referred to as the \emph{link of $V$}.
Topologically, this represents the boundary of a small neighbourhood
of $V$ in the triangulation.

A \emph{3-manifold triangulation} is a combinatorial triangulation that,
when viewed as a topological space, represents a 3-manifold.
Equivalently, a 3-manifold triangulation is a combinatorial
triangulation in which:
\begin{enumerate}[(i)]
    \item each vertex link is a topological sphere;
    \item no tetrahedron edge is identified with itself in reverse as a
    result of the face identifications.\footnote{%
        An equivalent condition to (ii)
        is that we can direct the edges of every
        tetrahedron in a manner consistent with the face identifications.}
\end{enumerate}

The earlier example is \emph{not} a 3-manifold triangulation,
since the link of the (unique) vertex is a torus, not a sphere.
% The underlying 3-manifold is connected if and only if the triangulation
% is connected (as defined above).
For many 3-manifolds $M$, the size of a minimal triangulation of $M$
corresponds to the Matveev complexity of $M$
\cite{martelli02-decomp,matveev03-algms}.

% Most of the work in enumeration algorithms involves partial
% triangulations.
A \emph{partial triangulation} is a combinatorial
triangulation in which we identify only $2k$ of the $4n$ tetrahedron
faces in pairs, for some $0 \leq 2k \leq 4n$.  We define vertices,
edges and vertex links as before, noting that vertex links might
now be surfaces with boundary (not closed surfaces).

A typical enumeration algorithm works as follows
\cite{burton07-nor10,hildebrand89-cuspedcensusold}:

\begin{algorithm}
Suppose we wish to
enumerate all connected 3-manifold triangulations of size $n$ satisfying
some set of properties $P$.  The main steps are:
\begin{enumerate}
    \item Enumerate all possible face pairing graphs (i.e., all
    connected 4-valent multigraphs on $n$ nodes).

    \item For each graph, recursively try all $6^{2n}$ possible
    rotations and reflections for identifying the corresponding
    tetrahedron faces.  Each ``partial selection'' of rotations and
    reflections gives a partial triangulation, and
    recursion and backtracking correspond to gluing and ungluing
    tetrahedron faces in these partial triangulations.

    \item Whenever we have a partial triangulation, run a series of
    tests that can identify situations where, no matter how we glue the
    remaining faces together, we can never obtain a 3-manifold triangulation
    satisfying the properties in $P$.  If this is the case, prune the
    current branch of the search tree and backtrack immediately.

    \item Whenever we have a complete selection of $6^{2n}$ rotations
    and reflections, test whether (i)~the corresponding combinatorial
    triangulation is in fact a 3-manifold triangulation, and
    (ii)~whether this 3-manifold triangulations satisfies the required
    properties in $P$.
\end{enumerate}
\end{algorithm}

There is also the problem of avoiding isomorphisms.
This is computationally cheap if the recursion is ordered carefully;
for details, see the full version of this paper.

In practice, step~1 is negligible---almost all of the computational work
is in the recursive search (steps 2--4).  The tests in step~3 are
critical: they must be extremely fast, since they are run at every stage of
the recursive search.  Moreover, if chosen carefully, these tests can prune
vast sections of the search tree and speed up the enumeration substantially.

A useful observation is that some graphs can be eliminated
immediately after step~1.
See \cite{burton04-facegraphs,burton07-nor10,martelli01-or9} for algorithms
that incorporate such techniques.

%%%%%%%%%%%%%%%%%%%%%%%%%%%%%%%%%%%%%%%%%%%%%%%%%%%%%%%%%%%%%%%%%%%%%%%%
%
%   Section:  Tracking Vertex Links
%
%%%%%%%%%%%%%%%%%%%%%%%%%%%%%%%%%%%%%%%%%%%%%%%%%%%%%%%%%%%%%%%%%%%%%%%%

\section{Tracking Vertex Links} \label{s-links}

In this paper we add the following test to step~3 of the enumeration
algorithm:

\begin{test} \label{test-genus}
    Whenever we have a partial triangulation $\tri$, test whether all
    vertex links are spheres with zero or more punctures.  If not,
    prune the current branch of the search tree and backtrack immediately.
\end{test}

Theoretically, it is simple to show that this test works:

\begin{lemma}
    If $\tri$ is a partial triangulation and the link of some vertex $V$
    is not a sphere with zero or more punctures, then
    there is no way to glue together the remaining faces of $\tri$
    to obtain a 3-manifold triangulation.
\end{lemma}

\begin{proof}
    Suppose we \emph{can} glue the remaining faces together to form a
    3-manifold triangulation $\tri'$.
    Let $L$ and $L'$ be the links of $V$ in $\tri$ and $\tri'$
    respectively; since $\tri'$ is a 3-manifold triangulation,
    $L'$ must be a topological sphere.

    This link $L'$ is obtained from $L$ by attaching zero or more additional
    triangles.  Therefore $L$ is an embedded subsurface
    of the sphere, and so $L$ must be a sphere with zero or more punctures.
\end{proof}

The test itself is straightforward; the difficulty lies in performing it
\emph{quickly}.  A fast implementation is crucial, since it will be
called repeatedly throughout the recursive search.

The key idea is to track the 1-dimensional \emph{boundary curves} of the
vertex links, which are formed from
cycles of edges belonging to vertex-linking triangles.
As we glue tetrahedron faces
together, we repeatedly split and splice these boundary cycles.
To verify Test~\ref{test-genus}, we must track which
triangles belong to the same vertex links and which edges belong to the
same boundary cycles, which we can do in expected logarithmic time
using union-find and skip lists respectively.

In the sections below, we describe what additional data needs to be
stored (Section~\ref{s-links-data}),
how to manipulate and use this data (Section~\ref{s-links-glue}),
and how skip lists can ensure a small time complexity
(Section~\ref{s-links-skip}).

\subsection{Data structures} \label{s-links-data}

In a partial triangulation with $n$ tetrahedra,
there are $4n$ \emph{vertex linking triangles} that together form the
vertex links (four such triangles are shown in Figure~\ref{fig-linktri}).
These $4n$ triangles are surrounded by a total of $12n$ \emph{vertex
linking edges}.

Each time we glue together two tetrahedron faces, we consequently
glue together three pairs of vertex linking edges, as shown in
Figure~\ref{fig-linkglue}.  This gradually combines the triangles
into a collection of larger triangulated surfaces, as illustrated in
Figure~\ref{fig-partiallink}.
The boundary curves of these surfaces are
drawn in bold in Figure~\ref{fig-partiallink};
these are formed from the vertex linking edges that have
not yet been paired together.

\begin{figure}[htb]
    \centering
    \includegraphics[scale=0.5]{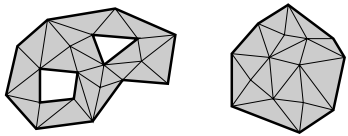}
    \caption{Vertex linking surfaces after gluing several tetrahedron
        faces}
    \label{fig-partiallink}
\end{figure}

To support Test~\ref{test-genus}, we store all $12n$ vertex linking
edges in a series of cyclic list structures that describe these boundary curves.
To simplify the discussion, we begin with a na\"ive implementation based on
doubly-linked lists.  However, this leaves us with an $O(n)$ operation
to perform, as seen in Section~\ref{s-links-glue}.
To run all operations in expected logarithmic time
we use skip lists \cite{pugh90-skiplists}, which we outline in
Section~\ref{s-links-skip}.

We treat the vertex linking edges as \emph{directed edges}
(i.e., arrows), with directions chosen arbitrarily at the beginning of
the enumeration algorithm.
For each vertex linking edge~$e$, we store the following data:
\begin{itemize}
    \item If $e$ is part of a boundary curve, we store the two edges
    adjacent to $e$ along this boundary curve, as well as two booleans
    that tell us whether these adjacent edges point in the same or opposite
    directions.

    \item If $e$ is not part of a boundary curve (i.e., it has been
    glued to some other vertex linking edge and is now internal to a
    vertex linking surface), we store a snapshot of the above data
    from the last time that $e$ \emph{was} part of a boundary curve.
\end{itemize}

To summarise: edges on the boundary curves are stored in a series of
doubly-linked lists, and internal edges
remember where they \emph{were} in these lists right before they were
glued to their current partner.

\subsection{Recursion, backtracking and testing} \label{s-links-glue}

Recall that each time we glue two tetrahedron faces together, we must
glue together \emph{three} pairs of vertex linking edges.  Each of these
edge gluings changes the vertex linking surfaces, and so we process each
edge gluing individually.

There are three key operations that we must perform in relation to edge
gluings:
\begin{enumerate}[(i)]
    \item gluing two vertex linking edges together (when we step forward
    in the recursion);
    \item ungluing two vertex linking edges (when we backtrack);
    \item verifying Test~\ref{test-genus} after gluing two vertex linking
    edges together (i.e., verifying that all vertex links are spheres
    with zero or more punctures).
\end{enumerate}

We now present the details of each operation in turn.  Throughout this
discussion we assume that edges are glued together so that all vertex
links are \emph{orientable} surfaces; the paper \cite{burton07-nor10}
describes an efficient framework for detecting non-orientable vertex
links as soon as they arise.

\subsubsection*{Recursion: gluing edges together}

Suppose we wish to glue together edges $x$ and $y$, as illustrated in
Figure~\ref{fig-glueedges}.  Only local modifications are required:
edges $p$ and $r$ become adjacent and must now to link to each other
(instead of to $x$ and $y$); likewise, edges $q$ and $s$ must be adjusted
to link to each other.  Note that this gluing introduces a
change of direction where $p$ and $r$ meet (and likewise for $q$ and $s$),
so as we adjust the direction-related booleans we must perform an extra
negation on each side.

\begin{figure}[htb]
    \centering
    \includegraphics[scale=0.9]{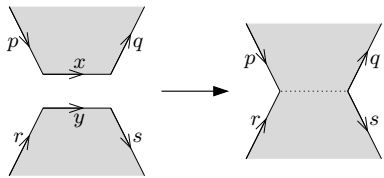}
    \caption{Gluing two vertex linking edges together}
    \label{fig-glueedges}
\end{figure}

We make no changes to the data stored for edges $x$ and $y$, since
these two edges are now internal and their associated data now represents
a snapshot from the last time that they were boundary edges (as
required by Section~\ref{s-links-data}).

All of these local modifications can be performed in $O(1)$ time.
The associated list operations are deletion, splitting and splicing;
this becomes important when we move to skip lists in
Section~\ref{s-links-skip}.

It is important to remember the special case in which edges $x$ and $y$
are adjacent in the same boundary cycle.  Here the local
modifications are slightly different (there are only two or possibly
zero nearby edges to
update instead of four), but these modifications remain $O(1)$ time.

\subsubsection*{Backtracking: ungluing edges}

As with gluing, ungluing a pair of vertex linking edges is a
simple matter of local modifications.  Here the backtracking context is
important: it is essential that we unglue edges in the reverse order to
that in which they were glued.

Suppose we are ungluing edges $x$ and $y$ as depicted in
Figure~\ref{fig-glueedges}.  The snapshot data stored with edges $x$ and
$y$ shows that they \emph{were} adjacent to edges $p$, $q$, $r$ and $s$
immediately before this gluing was made (and therefore immediately
\emph{after} the ungluing that we are now performing).

This snapshot data therefore gives us access to edges
$p$, $q$, $r$ and $s$: now we simply
adjust $p$ and $q$ to link to $x$ (instead of $r$ and $s$), and likewise
we adjust $r$ and $s$ to link to $y$.  No modifications to edges $x$ and
$y$ are required.

Again we must adjust the direction-related booleans carefully, and we
must cater for the case in which edges $x$ and $y$ were adjacent
immediately before the gluing was made.

As before, all local modifications can be performed in $O(1)$ time.
For the skip list discussion in Section~\ref{s-links-skip},
the associated list operations are splitting, splicing and insertion.

\subsubsection*{Testing: verifying that links are punctured spheres}

Each time we glue two vertex linking edges together we must ensure that
every vertex link is a sphere with zero or more punctures
(Test~\ref{test-genus}).
We test this \emph{incrementally}: we assume this is true
\emph{before} we glue these edges
together, and we verify that our new gluing does not introduce any
unwanted genus to the vertex links.

Our incremental test is based on the following two observations:

\begin{lemma} \label{l-genus-twospheres}
    Let $S$ and $S'$ be distinct triangulated spheres with punctures,
    and let $x$ and $y$ be boundary edges from $S$ and $S'$
    respectively.  If we glue $x$ and $y$ together, the resulting
    surface is again a sphere with punctures.
\end{lemma}

\begin{lemma} \label{l-genus-onesphere}
    Let $S$ be a triangulated sphere with punctures, and let $x$ and $y$
    be distinct boundary edges of $S$.  If we glue $x$ and $y$ together
    in an orientation-preserving manner, the resulting surface is a sphere
    with punctures if and only if $x$ and $y$ belong to the same
    boundary cycle of $S$.
\end{lemma}

\begin{figure}[htb]
    \centering
    \subfigure[Two distinct spheres with punctures]{%
        \label{fig-genus-twospheres}%
        \includegraphics[scale=0.88]{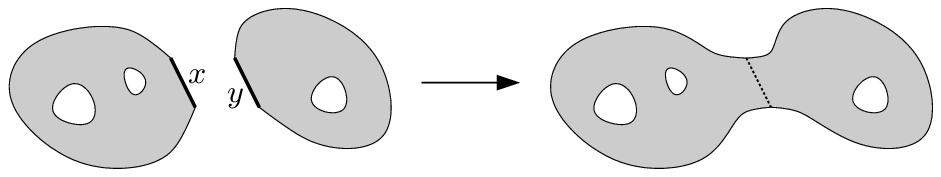}} \\
    \subfigure[Same sphere, different boundary cycles]{%
        \label{fig-genus-twocycles}%
        \includegraphics[scale=0.9]{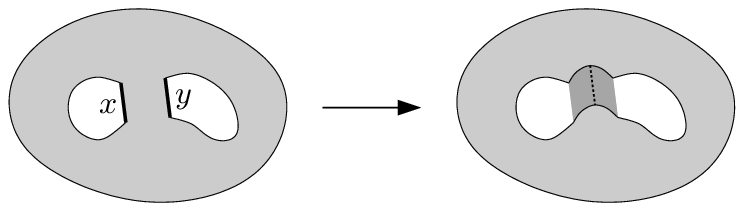}} \\
    \subfigure[Same boundary cycle, non-adjacent edges]{%
        \label{fig-genus-onecycle}%
        \includegraphics[scale=0.9]{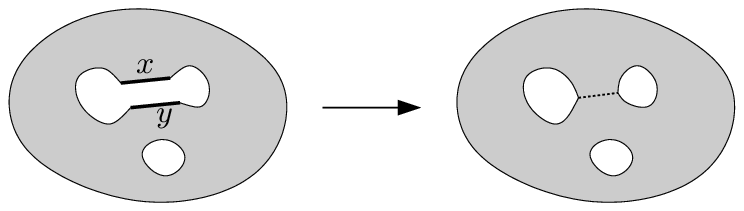}} \\
    \subfigure[Same boundary cycle, adjacent edges]{%
        \label{fig-genus-onecycleadj}%
        \includegraphics[scale=0.9]{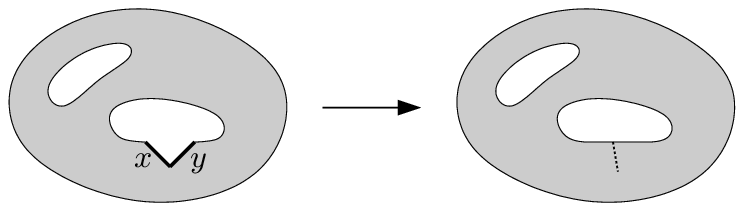}}
    \caption{Different ways of gluing boundary edges together}
\end{figure}

The proofs of Lemmata~\ref{l-genus-twospheres} and \ref{l-genus-onesphere}
are simple, and we do not give further details here.
Figure~\ref{fig-genus-twospheres} illustrates the scenario of
Lemma~\ref{l-genus-twospheres} with two distinct punctured spheres,
and Figures \ref{fig-genus-twocycles}--\ref{fig-genus-onecycleadj} show
different scenarios from Lemma~\ref{l-genus-onesphere} in which
we join two boundary edges from the same punctured sphere.

In particular,
Figure~\ref{fig-genus-twocycles} shows the case in which $x$ and $y$ belong
to different boundary cycles; here we observe that the resulting surface
is a twice-punctured torus.  Figures~\ref{fig-genus-onecycle}
and~\ref{fig-genus-onecycleadj} show cases with $x$ and $y$ on
the same boundary cycle; in \ref{fig-genus-onecycleadj},
$x$ and $y$ are adjacent along the boundary.

Note that Lemmata~\ref{l-genus-twospheres} and \ref{l-genus-onesphere}
hold even with very short boundaries (for instance, one-edge boundaries
consisting of $x$ or $y$ alone).

It is now clear how to incrementally verify Test~\ref{test-genus} when
we glue together vertex linking edges $x$ and $y$:
\begin{enumerate}
    \item
    Test whether the vertex linking triangles containing $x$ and
    $y$ belong to the same connected vertex linking surface.
    If not, the test passes.  Otherwise:

    \item
    Test whether $x$ and $y$ belong to the same doubly-linked list
    of boundary edges (i.e., the same boundary cycle).  If so, the test
    passes.  If not, the test fails.
\end{enumerate}

Here we implicitly assume that all gluings are
orientation-preserving, as noted at the beginning of
Section~\ref{s-links-glue}.

Step~1 can be performed in $O(\log n)$ time using the modified union-find
structure outlined in \cite{burton07-nor10}.
The original purpose of this structure was to enforce orientability in
vertex linking surfaces, and one of the operations it provides is an
$O(\log n)$ test for whether two vertex linking triangles belong to the
same connected vertex linking surface.  This modified union-find
supports backtracking; see \cite{burton07-nor10} for further details.

Step~2 is more difficult: a typical implementation might involve walking
through the doubly-linked list containing $x$ until we either find $y$
or verify that $y$ is not present, which takes $O(n)$ time to complete.
Union-find cannot help us, because of our repeated splitting and splicing
of boundary cycles.
In the following section we show how to reduce this $O(n)$ running time to
expected $O(\log n)$ by extending our doubly-linked lists to become
\emph{skip lists}.

It should be noted that Step~2 can in fact be carried out in $O(b)$
time, where $b$ is the number of boundary edges on all vertex linking
surfaces.  Although $b \in O(n)$ in general, for some face pairing
graphs $b$ can be far smaller.  For instance, when the face pairing
graph is a double-ended chain \cite{burton04-facegraphs}, we can arrange
the recursive search so that $b \in O(1)$.
See the full version of this paper for details.

\subsection{Skip lists and time complexity} \label{s-links-skip}

From the discussion in Section~\ref{s-links-glue}, we see that with
our na\"ive doubly-linked list implementation, the three
key operations of gluing edges, ungluing edges and verifying
Test~\ref{test-genus} have $O(1)$, $O(1)$ and $O(n)$ running times
respectively.
The bottleneck is the $O(n)$ test for whether two vertex linking edges
belong to the same doubly-linked list (i.e., the same boundary cycle).

We can improve our situation by extending our doubly-linked list to a
\emph{skip list} \cite{pugh90-skiplists}.  Skip lists are essentially
linked lists with additional layers of ``express pointers'' that allow us to
move quickly through the list instead of stepping forward one element
at a time.  Figure~\ref{fig-skiplist} shows a typical skip list structure.

\begin{figure}[htb]
    \centering
    \includegraphics[scale=0.5]{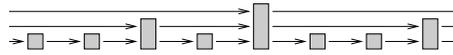}
    \caption{The internal layout of a skip list}
    \label{fig-skiplist}
\end{figure}

The list operations used in Section~\ref{s-links-glue} for gluing and ungluing
edges are deletion, insertion, splitting and splicing; all of these
can be performed on a skip list in expected $O(\log n)$ time
\cite{pugh90-cookbook,pugh90-skiplists}.
Importantly, it also takes expected
$O(\log n)$ time to search forward to the last
element of a skip list.\footnote{%
    Although our lists are cyclic, we can always define an arbitrary
    endpoint.}
We can therefore test whether vertex linking edges $x$ and $y$ belong to
the same list by searching forward to the end of each list
and testing whether the final elements are the same.

It follows that, with a skip list implementation, all three key operations
of gluing edges, ungluing edges and verifying Test~\ref{test-genus}
run in expected $O(\log n)$ time.  Therefore:

\begin{theorem}
    It is possible to implement Test~\ref{test-genus} by performing
    expected $O(\log n)$ operations at every stage of the recursive search.
\end{theorem}

Note that we must include reverse links at the lowest layer
of the skip list (effectively maintaining the original doubly-linked list),
since both forward and backward links are required for the
gluing and ungluing operations (see Section~\ref{s-links-glue}).
Full details of the skip list implementation can be found in the full
version of this paper.

%%%%%%%%%%%%%%%%%%%%%%%%%%%%%%%%%%%%%%%%%%%%%%%%%%%%%%%%%%%%%%%%%%%%%%%%
%
%   Section:  Performance
%
%%%%%%%%%%%%%%%%%%%%%%%%%%%%%%%%%%%%%%%%%%%%%%%%%%%%%%%%%%%%%%%%%%%%%%%%

\section{Performance} \label{s-perf}

Here we measure the performance of our new algorithm experimentally.
Specifically, we compare two enumeration algorithms:
the \emph{old algorithm}, which includes all of the optimisations
described in \cite{burton07-nor10} (including the union-find framework
for ensuring that vertex links remain orientable), and the
\emph{new algorithm}, which enhances the old algorithm with the
new tests described in Section~\ref{s-links}.\footnote{%
    We compare the new algorithm against \cite{burton07-nor10} because
    this allows us to isolate our new techniques, and because
    the source code and implementation details for alternative
    algorithms \cite{martelli01-or9,matveev03-algms} are not readily available.}

We run our performance tests by enumerating censuses of all
3-manifold triangulations of size $n \leq 7$.  We choose this type of
census because a census of minimal triangulations requires significant
manual post-processing \cite{burton07-nor10}, and because a census of
all triangulations is significantly larger and therefore a stronger
``stress test''.  We restrict our tests to $n \leq 7$ because for larger
$n$ the old algorithm becomes too slow to run time trials on a single
CPU.

\begin{table}[htb]
\centering
\small
\begin{tabular}{l|r|r|r|r}
\multicolumn{1}{c|}{Census parameters} &
\multicolumn{1}{c|}{Triangulations} &
\multicolumn{1}{c|}{Old (h:m:s)} &
\multicolumn{1}{c|}{New (m:s)} &
\multicolumn{1}{c}{Speed-up} \\
\hline
$n=5$, orientable  &   4\,807 &     0:59 &     0:03 & $20\times$ \\
$n=5$, non-orientable &      377 &     1:09 &     0:06 & $12\times$ \\[1ex]
$n=6$, orientable  &  52\,946 &  1:11:57 &     1:03 & $69\times$ \\
$n=6$, non-orientable &   4\,807 &  1:23:30 &     2:05 & $40\times$ \\[1ex]
$n=7$, orientable  & 658\,474 & 92:23:39 &    22:16 & $249\times$ \\
$n=7$, non-orientable &  64\,291 &103:24:51 &    48:27 & $128\times$
\end{tabular}
\caption{Running times for old and new algorithms}
\label{tab-perf}
\end{table}

Table~\ref{tab-perf} shows the results of our time trials,
split into censuses of orientable and non-orientable triangulations.
For $n \leq 4$ both algorithms run in 1~second or less.
All trials were carried out on a single 2.93\,GHz Intel Xeon X5570 CPU.

The results are extremely pleasing:
for $n=7$ we see a speed-up of $128\times$ in the non-orientable case
and $249\times$ in the orientable case (from almost four days of
running time down to just 22~minutes).
Moreover, the speed-up factors appear to grow exponentially with $n$.
All of this suggests that our new algorithm can indeed make a concrete
difference as to how large a census we can feasibly build.

It is worth noting that speed-ups are consistently better for the
orientable case.  This may be because the union-find framework introduced in
\cite{burton07-nor10} is most effective for non-orientable enumeration
(as noted in Section~\ref{s-intro}), and so the orientable case has
more room for gain.  Nevertheless, it is pleasing to see that the
new algorithm gives substantial improvements for both the orientable and
non-orientable cases.

%%%%%%%%%%%%%%%%%%%%%%%%%%%%%%%%%%%%%%%%%%%%%%%%%%%%%%%%%%%%%%%%%%%%%%%%
%
%   Section:  Parallelisation
%
%%%%%%%%%%%%%%%%%%%%%%%%%%%%%%%%%%%%%%%%%%%%%%%%%%%%%%%%%%%%%%%%%%%%%%%%

\section{Parallelisation} \label{s-par}

As $n$ increases, the output size for a typical census grows
exponentially in $n$, and sometimes super-exponentially---for instance,
the growth rate of a census of all 3-manifold triangulations is
known to be $\exp(\Theta(n\log n))$ \cite{burton10-pachner}.
It is therefore critical that enumeration algorithms be parallelised if
we are to make significant progress in obtaining new census data.

Like many combinatorial searches, the enumeration of
triangulations is an embarrassingly parallel problem: different branches
of the search tree can be processed independently, making the
problem well-suited for clusters and server farms.  Avoiding
isomorphisms causes some minor complications, which we discuss in the
full version of this paper.

The main obstacle is that, because of the various pruning techniques
(as described in Sections~\ref{s-prelim} and~\ref{s-links}), it is very
difficult to estimate in advance how long each branch of the search tree
will take to process.  Experience shows that there can be
orders-of-magnitude differences in running time between subsearches at
the same depth in the tree.

For this reason, parallelisation must use a controller\,/\,slave model
in which a controller process repeatedly hands small pieces of the search
space to the next available slave, as opposed to a simple subdivision in
which each process handles a fixed portion of the search space.
This means that some inter-process communication is required.

For each subsearch, the controller must send $O(n)$ data to the
slave: this includes the face pairing graph, the partial triangulation,
and the data associated with the vertex linking edges and triangles
as described in Section~\ref{s-links}.  The output for each
subsearch can be super-exponentially large, and so it is preferable for
slaves to write this data directly to disk (as opposed to communicating
it back to the controller).  Collating the output from different
slaves is a simple task that can be performed after the enumeration
has finished.

The enumeration code in {\regina} implements such a model using
MPI, and runs successfully on hundreds of simultaneous CPUs with
a roughly proportional speed-up in wall time.

%%%%%%%%%%%%%%%%%%%%%%%%%%%%%%%%%%%%%%%%%%%%%%%%%%%%%%%%%%%%%%%%%%%%%%%%
%
%   Section:  Census Data
%
%%%%%%%%%%%%%%%%%%%%%%%%%%%%%%%%%%%%%%%%%%%%%%%%%%%%%%%%%%%%%%%%%%%%%%%%

\section{Census Data} \label{s-data}

The new algorithms in this paper have been implemented and run
in parallel using
high-performance computing facilities to obtain new census data
that exceeds the best known limits in the literature.  This includes
(i)~a census of all closed 3-manifold triangulations of size $n \leq 10$,
and (ii)~a census of all minimal triangulations of closed {\ppirr}
3-manifolds of size $n \leq 11$.

\subsection{All closed 3-manifold triangulations}

The first reported census of all closed 3-manifold
triangulations appears in \cite{burton10-complexity} for $n \leq 9$,
and has been used to study algorithmic complexity
and random triangulations \cite{burton10-complexity,burton10-pachner}.
Here we extend this census to $n \leq 10$ with a total of over
2~billion triangulations:

\begin{theorem}
    There are precisely $2\,196\,546\,921$
    closed 3-manifold triangulations that can be constructed from
    $\leq 10$ tetrahedra, as summarised by Table~\ref{tab-all}.
\end{theorem}

\begin{table}[htb]
\centering
\small
\begin{tabular}{c|r|r|r}
Size ($n$) & Orientable & Non-orientable & Total \\
\hline
1   &                4 &              --- &                4 \\
2   &               16 &                1 &               17 \\
3   &               76 &                5 &               81 \\
4   &              532 &               45 &              577 \\
5   &           4\,807 &              377 &           5\,184 \\
6   &          52\,946 &           4\,807 &          57\,753 \\
7   &         658\,474 &          64\,291 &         722\,765 \\
8   &      8\,802\,955 &         984\,554 &      9\,787\,509 \\
9   &    123\,603\,770 &     15\,499\,262 &    139\,103\,032 \\
10  & 1\,792\,348\,876 &    254\,521\,123 & 2\,046\,869\,999 \\
\hline
Total&1\,925\,472\,456 &    271\,074\,465 & 2\,196\,546\,921
\end{tabular}
\caption{All closed 3-manifold triangulations}
\label{tab-all}
\end{table}

The total CPU time required to enumerate the 10-tetra\-hedron
census was $\simeq 2.4$ years, divided amongst
192 distinct 2.93\,GHz Intel Xeon X5570 CPUs.

The paper \cite{burton10-complexity} makes two conjectures regarding
the worst-case and average number of vertex normal surfaces
for a closed 3-manifold triangulation of size $n$.  Details and
definitions can be found in \cite{burton10-complexity}; in summary,
these conjectures are:

\begin{conjecture} \label{conj-normal-worst}
    For all positive $n \neq 1,2,3,5$, a tight upper bound on the
    number of vertex normal surfaces in a closed 3-manifold
    triangulation of size $n$ is:
    \[ \begin{array}{l@{\quad}l}
        17^k + k & \mbox{if\ $n=4k$;} \\
        581 \cdot 17^{k-2}+k+1 & \mbox{if\ $n=4k+1$;} \\
        69 \cdot 17^{k-1}+k & \mbox{if\ $n=4k+2$;} \\
        141 \cdot 17^{k-1}+k+2 & \mbox{if\ $n=4k+3$,}
    \end{array} \]
    and so this upper bound grows asymptotically as $\Theta(17^{n/4})$.
\end{conjecture}

\begin{conjecture} \label{conj-normal-avg}
    If $\overline{\sigma}_n$ represents
    the average number of vertex normal surfaces amongst all closed
    3-manifold triangulations (up to isomorphism), then
    $\overline{\sigma}_n < \overline{\sigma}_{n-1} +
    \overline{\sigma}_{n-2}$
    for all $n \geq 3$, and so
    $\overline{\sigma}_n \in O([\frac{1+\sqrt{5}}{2}]^n)$.
\end{conjecture}

These conjectures were originally based on the census data for
$n \leq 9$.  With our new census we can now verify these
conjectures at the 10-tetrahedron level:

\begin{theorem}
    Conjectures~\ref{conj-normal-worst} and~\ref{conj-normal-avg}
    are true for all $n \leq 10$.
\end{theorem}

This census contains over 63\,GB of data, and so the data files have not
been posted online.  Readers who wish to work with
this data are welcome to contact the author for a copy.

\subsection{Closed {\ppirr} 3-manifolds}

A 3-manifold is \emph{\ppirr} if every sphere bounds a ball and
there are no embedded two-sided projective planes.
Censuses of closed {\ppirr}
3-manifolds and their minimal triangulations
have a long history
\cite{amendola03-nor6,amendola05-nor7,burton07-nor10,burton07-nor8,
martelli06-or10,martelli01-or9,matveev05-or11,
matveev05-or12,matveev88-hamiltonian}.
The largest reported census of all minimal \emph{triangulations}
of these manifolds reaches $n\leq 10$ \cite{burton07-nor10}.  If we enumerate
\emph{manifolds} but not their triangulations, the censuses reaches
$n\leq 12$ in the orientable case \cite{matveev05-or12} but remains at
$n\leq 10$ in the non-orientable case.

Here we extend this census of minimal triangulations of closed
{\ppirr} 3-manifolds to $n\leq 11$.
As a result, we also extend the census of underlying
manifolds to $n\leq 11$ in the non-orientable case, and in the
orientable case we confirm that the number of manifolds matches
Matveev's census \cite{matveev05-or12}.

\begin{theorem}
    There are precisely $13\,765$
    closed {\ppirr} 3-manifolds that can be constructed from
    $\leq 11$ tetrahedra.  These have a combined total of
    $55\,488$ minimal triangulations,
    as summarised by Table~\ref{tab-minimal}.
\end{theorem}

\begin{table}[htb]
\centering
\small
\begin{tabular}{c|r|r|r|r}
Size & \multicolumn{2}{c|}{Minimal triangulations} &
\multicolumn{2}{c}{Distinct 3-manifolds} \\
($n$) & Orientable & Non-orient. & Orientable & Non-orient. \\
\hline
 1 &       4 &    --- &       3 & --- \\
 2 &       9 &    --- &       6 & --- \\
 3 &       7 &    --- &       7 & --- \\
 4 &      15 &    --- &      14 & --- \\
 5 &      40 &    --- &      31 & --- \\
 6 &     115 &     24 &      74 &   5 \\
 7 &     309 &     17 &     175 &   3 \\
 8 &     945 &     59 &     436 &  10 \\
 9 &  3\,031 &    307 &  1\,154 &  33 \\
10 & 10\,244 &    983 &  3\,078 &  85 \\
11 & 36\,097 & 3\,282 &  8\,421 & 230 \\
\hline
Total & 50\,816 & 4\,672 & 13\,399 & 366
\end{tabular}
\caption{All minimal triangulations of closed {\ppirr} 3-manifolds}
\label{tab-minimal}
\end{table}

The paper \cite{burton07-nor8} raises conjectures for certain classes of
non-orientable 3-manifolds regarding the combinatorial structure of every
minimal triangulation.
Again we refer to the source \cite{burton07-nor8}
for details and definitions; in summary:

\begin{conjecture} \label{conj-min-torus}
    Every minimal triangulation of a non-flat non-orientable torus
    bundle over the circle is a \emph{layered torus bundle}.
\end{conjecture}

\begin{conjecture} \label{conj-min-plugged}
    Every minimal triangulation of a non-flat non-orientable
    Seifert fibred space over $\R P^2$ or $\bar{D}$ with
    two exceptional fibres is either a \emph{plugged thin $I$-bundle} or a
    \emph{plugged thick $I$-bundle}.
\end{conjecture}

Layered torus bundles and plugged thin and thick $I$-bun\-dles
are families of triangulations with well-defined
combinatorial structures.
The original conjectures were based on census data for $n \leq 8$, and
in \cite{burton07-nor10} they are shown to hold for all $n \leq 10$.
With our new census data we are now able to validate
these conjectures at the 11-tetrahedron level:

\begin{theorem}
    Conjectures~\ref{conj-min-torus} and~\ref{conj-min-plugged}
    are true for all minimal triangulations of size $n \leq 11$.
\end{theorem}

Data files for this census, including the 3-manifolds and all of
their minimal triangulations, can be downloaded from the {\regina} website
\cite{regina}.

%%%%%%%%%%%%%%%%%%%%%%%%%%%%%%%%%%%%%%%%%%%%%%%%%%%%%%%%%%%%%%%%%%%%%%%%
%
%   Acknowledgements
%
%%%%%%%%%%%%%%%%%%%%%%%%%%%%%%%%%%%%%%%%%%%%%%%%%%%%%%%%%%%%%%%%%%%%%%%%

\section*{Acknowledgments}

Computational resources used in this work
were provided by the Queensland Cyber Infrastructure Foundation
and the Victorian Partnership for Advanced Computing.

%%%%%%%%%%%%%%%%%%%%%%%%%%%%%%%%%%%%%%%%%%%%%%%%%%%%%%%%%%%%%%%%%%%%%%%%
%
%   Bibliography
%
%%%%%%%%%%%%%%%%%%%%%%%%%%%%%%%%%%%%%%%%%%%%%%%%%%%%%%%%%%%%%%%%%%%%%%%%

\small
\bibliographystyle{amsplain}
\bibliography{pure}

%%%%%%%%%%%%%%%%%%%%%%%%%%%%%%%%%%%%%%%%%%%%%%%%%%%%%%%%%%%%%%%%%%%%%%%%
%
%   Author contacts and affiliation
%
%%%%%%%%%%%%%%%%%%%%%%%%%%%%%%%%%%%%%%%%%%%%%%%%%%%%%%%%%%%%%%%%%%%%%%%%

\bigskip
\noindent
Benjamin A.~Burton \\
School of Mathematics and Physics, The University of Queensland \\
Brisbane QLD 4072, Australia \\
(bab@maths.uq.edu.au)

%%%%%%%%%%%%%%%%%%%%%%%%%%%%%%%%%%%%%%%%%%%%%%%%%%%%%%%%%%%%%%%%%%%%%%%%
%
%   Appendix
%
%%%%%%%%%%%%%%%%%%%%%%%%%%%%%%%%%%%%%%%%%%%%%%%%%%%%%%%%%%%%%%%%%%%%%%%%

\normalsize
\appendix

\section*{Appendix}

In the introduction we claim there are at least $6.44 \times 10^{12}$
connected combinatorial triangulations of size $n=9$, up to isomorphism.
Here we give the arguments to support this claim.

We begin by placing a lower bound on the number of \emph{labelled}
connected combinatorial triangulations.
To ensure that each triangulation is connected, we insist that the first face
of tetrahedron~$k$ is glued to some face chosen from tetrahedra
$1,\ldots,k-1$, for all $k > 1$.  Of course there are many labelled connected
triangulations that do not satisfy this constraint, but since we are
computing a lower bound this does not matter.

We initially choose gluings for the first face of each tetrahedron
$2,3,\ldots,n$ in order.  For the first face of tetrahedron $k$ there are
$2k$ choices for a partner face---these are the $4(k-1)$ faces of tetrahedra
$1,\ldots,k-1$ minus the $2(k-2)$ faces already glued---as well as
six choices of rotation or reflection.
This gives a total of
$4 \times 6 \times \ldots \times (2n-2) \times 2n \times 6^{n-1}$
possibilities.  From here there are
$(2n+1) \times (2n-1) \times \ldots \times 3 \times 1 \times 6^{n+1}$
ways of gluing together the remaining $2n+2$ faces in pairs,
giving a lower bound of at least
$(2n+1)! \times 6^{2n}/2$ labelled connected combinatorial triangulations
of size $n$.

We finish by factoring out isomorphisms.  Each isomorphism class has
size at most $n! \times {4!}^n$ (all possible relabellings of tetrahedra and
their vertices), and so the total number of connected
combinatorial triangulations of size $n$ \emph{up to isomorphism} is at least
\[ \frac{(2n+1)! \times 6^{2n}}{2 \times n! \times {4!}^n} .\]
For $n=9$ this evaluates to approximately
$6.4435 \times 10^{12}$.

\end{document}